\documentclass[11pt,twoside,a4paper]{amsart}
\usepackage{amssymb}
\usepackage{eufrak} 
\date{\today}
\usepackage{geometry}
\def\w{\wedge}

\def\dbar{\bar\partial}

\def\C{{\mathbb C}}
\def\w{{\wedge}}

\def\F{{\mathcal F}}

\def\codim{{\rm codim\,}}

\def\U{{\mathcal U}}

\def\J{{\mathcal J}}
\def\nbh{neighborhood }

\def\be{\begin{equation}}
\def\ee{\end{equation}}

\def\Ok{\mathcal O}
\def\mult{{\rm mult}}

\def\1{{\bf 1}}

\def\m{{\mathfrak m}}

\newtheorem{thm}{Theorem}[section]
\newtheorem{lma}[thm]{Lemma}
\newtheorem{cor}[thm]{Corollary}
\newtheorem{prop}[thm]{Proposition}

\theoremstyle{definition}

\theoremstyle{remark}

\newtheorem{preremark}[thm]{Remark}
\newtheorem{preex}[thm]{Example}

\newenvironment{remark}{\begin{preremark}}{\qed\end{preremark}}
\newenvironment{ex}{\begin{preex}}{\qed\end{preex}}

\numberwithin{equation}{section}

\title[Green functions,  Segre numbers, and King's formula, 
]
{Green functions, Segre numbers, \\ and King's formula }

\begin{document}
\date{\today}

\author[M.\  Andersson \& E.\  Wulcan]
{Mats  Andersson \& Elizabeth Wulcan}

\address{Department of Mathematics\\Chalmers University of Technology and the University of Gothenburg\\S-412 96 
Gothenburg\\SWEDEN}

\email{matsa@chalmers.se, wulcan@chalmers.se}





\thanks{The authors were partially supported by the Swedish 
  Research Council}


\begin{abstract}  Let $\J$ be a coherent ideal sheaf 
on a complex manifold $X$
with zero set $Z$, and let $G$ be a plurisubharmonic function such that
$G=\log|f|+\Ok(1)$ locally at $Z$, where $f$ is a tuple of holomorphic
functions that defines $\J$. We give a meaning to the Monge-Amp\`{e}re
products $(dd^c G)^k$ for $k=0,1,2,\ldots$, and prove that the Lelong
numbers of the currents  $M_k^{\J}:=\1_Z(dd^c G)^k$ at $x$
coincide with the so-called Segre numbers of
$\J$ at $x$, introduced independently by Tworzewski, 
Gaffney-Gassler, and Achilles-Manaresi.
More generally, we show that  $M_k^{\J}$
satisfy a certain generalization of the classical King formula.

\medskip 



\end{abstract}

\maketitle

\section{Introduction}

Let $X$ be a complex manifold of dimension $n$
and let $\J\to X$ be a coherent ideal sheaf with variety $Z$. 
Given a point $x\in X$,  
Tworzewski, \cite{T},
and Gaffney and Gassler, \cite{GG},  have independently introduced a list of
numbers, $e_0(\J,X,x),\ldots,e_n(\J,X,x)$, that we,
following \cite{GG}, call the {\it Segre numbers} at $x$. They are
a generalization of the classical local intersection number at $x$ in case
the ideal $\J_x$ is a complete intersection. 
The definition in both papers is based on a local variant of the
St\"uckrad-Vogel procedure, \cite{SV}.  In \cite{AM, AR} is given an algebraic
definition of these numbers generalizing the classical 
Hilbert-Samuel multiplicity of $\J$ at $x$.

In this paper we show that if $\J$ is generated by global bounded
functions there is a canonical global representation of the
Segre numbers of $\J$ as the Lelong numbers (of restrictions to $Z$) of Monge-Amp\`{e}re masses of the \emph{Green function $G=G_\J$ with poles along $\J$}. 
This function was introduced by Rashkovskii-Sigurdsson in
\cite[Definition~2.2]{RS} as a generalization of the classical Green
function $G_a$ with pole at a point $a\in X$. It is defined as the supremum
over the class $\F_\J$ of all negative psh (plurisubharmonic) functions $u$
on $X$ that locally satisfy $u\leq \log|f| + C$, where
$f=(f_1,\ldots, f_m)$ is a tuple of local generators of $\J$ and $C$
is a constant.

Note that even if $X$ is hyperconvex there might not exist non-trivial
functions in $\F_\J$. For example, if $X$ is the ball in $\C$, and
$\J$ is the radical ideal of functions vanishing at points
$a_1,a_2,\ldots\in X$, then there are negative
psh functions with poles at $a_j$ if and only if $a_j$
satisfy the the Blaschke condition.  
However, if $\J$ is globally generated by bounded functions $f_j$,
then $\log |f|+C$ is itself in $\F_\J$ for some constant $C$. 
Then locally $G$ is of the form 
\begin{equation}\label{formen}
G=\log|f|+h,
\end{equation} 
where
$h$ is locally bounded, see \cite[Theorem~2.8]{RS}. In particular, the unbounded locus of $G$ equals $Z$
and thus the Monge-Amp\`{e}re type products 
\begin{equation}\label{maprod}
(dd^c G)^k, \quad k\leq p :=\codim Z
\end{equation}
are well-defined, see, e.g., 
\cite[Theorem~III.4.5]{Dem}. Here and throughout 
$d^c=(i/2\pi)(\dbar -\partial)$.
By \emph{Demailly's comparison formula for Lelong numbers},  
\cite[Theorem~5.9]{Dem2}, 
\begin{equation}\label{demma}
\ell_x(dd^c G)^k = \ell_x(dd^c \log |f|)^k
\end{equation} 
for $x\in X$, where $\ell_x$ denotes the Lelong number at
$x$. 
Moreover, recall that \emph{King's formula}, \cite{King}, 
asserts that 
$(dd^c \log |f|)^p$ admits the Siu decomposition, \cite{Siu},
\begin{equation}\label{kingking}
(dd^c \log|f|)^p=\sum \beta_j [Z_j^p]+ R,
\end{equation} 
cf.\ \cite[Section~6]{Dem2}. 
Here $[Z_j^p]$ are the currents of integration along the irreducible
components $Z_j^p$ of codimension $p$ of $Z$, $\beta_j$ are the generic Hilbert-Samuel
multiplicities of $f$ along $Z_j^p$, see,
e.g. \cite[Chapter~4.3]{Fu}. 
In fact, the remainder term $R$ has integer Lelong numbers, see,
e.g. \cite[Theorem~1.1]{artikelaswy}, and therefore the set where $R$
has positive Lelong numbers is an analytic set of codimension $>p$. 
From \eqref{demma} and \eqref{kingking} one deduces that 
\begin{equation}\label{rssiu}
(dd^c G)^p=\sum \beta_j [Z_j^p]+R,
\end{equation} 
where $\beta_j$ and $Z_j^p$ are as above, and $R$ has the same Lelong
numbers as $R$ in \eqref{kingking}, cf.\ the proof of
Theorem~2.8 in \cite{RS}. In particular, if $Z$ is a point $a$, then 
$(dd^c G)^n=\sum \beta [a]+R$, where $[a]$ is the point evaluation at $a$
 and $\beta$ is the Hilbert-Samuel
  multiplicity of $\J$.
This generalizes the fact that $(dd^c G_a)^n=[a]$,
\cite[page~520]{Dem3}. 
The (Lelong numbers of the) Monge-Amp\`ere products \eqref{maprod} are related to the
integrability index of $G$ (and thus the 
log-canoncial threshold of $\J$), see, e.g., \cite{DP, Ras, Sk}; in
particular, Demailly-Pham \cite{DP} recently gave a sharp estimate of the
integrability index of $G$ in terms of the Lelong numbers of \eqref{maprod} for all $k\leq
p$. 

\medskip 
 Recall that \eqref{maprod}
can be defined inductively as 
\begin{equation}\label{forvirrad}
dd^c(G (dd^c G)^{k-1}).
\end{equation} 
In this paper we give meaning to $(dd^c G)^k$ for any $k$ if $G$ is
any psh function of the form \eqref{formen}: 
Inductively we show that 
\[
G\1_{X\setminus Z}(dd^c G)^{k-1}
\] 
has locally finite mass and define 
\[
(dd^c G)^k:=dd^c(G\1_{X\setminus Z}(dd^c G)^{k-1}),
\]  
see Proposition ~\ref{asgam}. 
When  $k\le p$ it follows from the dimension principle for closed positive currents, cf.\ Lemma~\ref{lma2} below,
that $\1_Z(dd^cG)^{k-1}=0$ 
and so our definition coincides with the classical one
for  $k\le p$. 
Our definition is modeled on the paper \cite{A2} by the first
author, in which currents $(dd^c\log |f|)^k$ are defined for all $k$
inductively as
above. In fact, $(dd^c\log |f|)^k$ can also be defined as a certain
limit of smooth forms coming from regularizations of $\log
|f|$:  
\begin{equation}\label{somna}
\lim_{\epsilon\to 0} (dd^c\log (|f|^2+\epsilon)^{1/2} )^k = (dd^c\log |f|)^k
\end{equation} 
for any $k$, 
see \cite[Proposition~4.4]{A2}. However, one cannot hope for such a suggestive definition of
$(dd^c G)^k$ in general, cf. Example ~\ref{exodus}. 
Also, our definition of $(dd^c G)^k$ does not coincide with the
\emph{non-pluripolar product} of $dd^cG$, as introduced in \cite{BT2,BEGZ}, since our $(dd^c G)^k$ charges pluripolar sets in
general, cf.\ the text after the proof of Proposition ~\ref{asgam}. 

Our main result is the following generalization of \eqref{rssiu}. 
Let $\pi^+\colon X^+\to X$ be the normalization of the blow-up of $X$
along $\J$ 
and let $W_j$ be the various irreducible components
of the exceptional divisor in $X^+$. Recall that the
\emph{(Fulton-MacPherson) distinguished varieties} of $\J$ are the
subvarieties $\pi^+(W_j)$ of  $X$, see, e.g., \cite[Chapter~10.5]{Lazar}. 
In particular, the distinguished varieties of codimension $p$ are precisely
the irreducible components of $Z$ of codimension $p$.

\begin{thm}\label{gking} 
Let $X$ be an $n$-dimensional complex manifold,
let  $\J$ be a coherent ideal sheaf on $X$ generated by global bounded
functions, and let $G$ be the Green function with poles along $\J$. 
Moreover, let $Z$ be the variety of $\J$ and $Z_j^k$ the Fulton-MacPherson distinguished varieties 
of $\J$ of codimension $k$.  
Then  
\begin{equation}\label{king}
M^\J_k:={\bf 1}_{Z} (dd^c G)^k=\sum_{j} \beta_j^k [Z_j^k]+N_k^\J =: S^\J_k+N^\J_k,  
\end{equation}
where the $\beta_j^k$ are positive integers and  the $N_k^\J$ are positive closed currents.
The numbers $n_k(\J,X,x):=\ell_x(N_k^\J)$
are nonnegative integers that only depend on the
integral closure class of $\J$ at $x$, and the set where 
$n_k(\J,X,x)\ge 1$ has codimension at least $k+1$.

The Lelong numbers at $x$ of $M^\J_k$ and  ${\bf 1}_{X\setminus
  Z}(dd^c G)^k$ are precisely the Segre number $e_k(\J,X,x)$ and the
polar multiplicity  $m_k(\J,X,x)$, respectively, of $\J_x$. 
\end{thm}
For the notion of polar multiplicities see Section \ref{segretal}. Notice that  $M_k^\J=0$ if $k<\codim Z$ and that $N_{p}^\J=0$, cf., 
Lemma~\ref{lma2} below.
Also, notice that  \eqref{king} is the Siu decomposition, \cite{Siu},   of $M_k^\J$.


\begin{remark}\label{gulpenna} 
If $\J$ is generated by a global tuple $f$, then Theorem \ref{gking} holds with
$G$ replaced by 
any psh function of the form \eqref{formen}. 
\end{remark}

\smallskip
The analogous statement to Theorem \ref{gking} when  $G$ is replaced by 
$\log|f|$, where $f$ is a tuple of global generators, was proved by the authors and Samuelsson and Yger
in \cite[Theorem~1.1]{artikelaswy}. 
The case $k=p$ corresponds to the classical King formula, \eqref{kingking}. 
The main idea in the proof of Theorem ~\ref{gking} is to prove that for
any psh $G$ of the form ~\eqref{formen}, 
\begin{equation}\label{jamfora}
\ell_x(\1_Z(dd^c G)^k)=\ell_x(\1_Z(dd^c\log |f|)^k), ~
\ell_x(\1_{X\setminus Z} (dd^c G)^k)=\ell_x(\1_{X\setminus Z}(dd^c\log |f|)^k)
\end{equation}
for $x\in X$, see Lemma \ref{torsdag} below. 
Using this the theorem follows from the corresponding result in
\cite{artikelaswy}. 
In some sense, \eqref{jamfora} can be seen as a generalization of Demailly's
comparison formula, \eqref{demma}, to higher $k$, but for the very
special class of psh functions of the form \eqref{formen}.

In \cite{artikelaswy}, $X$ is allowed to be singular. Given that
there is a proper definition of $G$ when $X$ is singular so that
\eqref{formen} still holds,
the results in this paper will extend as well.

\smallskip
Theorem \ref{gking} gives us a canonical representation of the Segre
numbers of $\J$ in the case when $\J$ is generated by global bounded
functions. 
Let $X$ be a, say hyperconvex, domain in $\C^n$, and let $\J$ be a coherent ideal sheaf on $X$.
If we exhaust $X$ by reasonable relatively compact subsets $X_\ell$, for each $\ell$ we then
have currents $M_k^{\J_\ell}$, $\J_\ell=\J|_{X_\ell}$, whose Lelong numbers  at each point are the Segre numbers. If
for some reason these currents converge to  currents $M_k^{\J}$, we would have a canonical
representation of the Segre numbers of $\J$ on $X$, cf.\ Remark
\ref{forsvinn}.


\smallskip 

This paper is organized as follows. In Section ~\ref{segretal} we
recall the construction of Vogel cycles and Segre numbers. In Section
\ref{avfall} we show that the currents $(dd^c G)^k$ are well-defined
and discuss some properties. The proof of Theorem ~\ref{gking} occupies
Section \ref{bevisas}. In Sections \ref{prelim} and \ref{kjol} we give
some background on psh functions and positive currents
needed for the proofs.

\subsection*{Acknowledgment}

The work on this paper started when Pascal Thomas was visiting
G\"oteborg. We are grateful to him for interesting and
inspiring discussions on the subject. 
We would also like to thank Zbigniew B\l{}ocki and David Witt Nystr\"om 
for valuable discussions.

\section{Segre numbers}\label{segretal}

We will briefly recall the construction of Segre numbers from
\cite{T,GG}. Throughout we will assume that $X$ is a complex manifold
of dimension $n$ and that $\J$ is a coherent ideal sheaf on $X$ with
variety $Z$. 
Fix a point $x\in X$. 
A sequence $h=(h_1,h_2,\ldots, h_n)$ in the local ideal $\J_x$ is called a 
{\it Vogel sequence of $\J$ at $x$}
 if  there is 
a \nbh $\U\subset X$ of $x$ where the $h_j$ are defined, such that
\begin{equation}\label{vogelcondition}
\codim\big[(\U\setminus Z)\cap(|H_1|\cap\cdots\cap |H_{k}|)\big]=k \ {\rm or} \ \infty, \ k=1,\ldots,n;
\end{equation}
here  $|H_\ell|$ are the supports of the
divisors $H_\ell$ defined by $h_\ell$. 
Notice that if $f_1,\ldots, f_m$ generate $\J_x$, 
any generic sequence of $n$ linear combinations of the $f_j$ is a Vogel sequence 
at $x$.  
Set $X_0=X$, let $X_0^Z$ denote the irreducible components of $X_0$
that are contained in $Z$,  
and let $X_0^{X\setminus Z}$ be the  remaining components\footnote{Since we assume $X$ is smooth and 
connected, $X_0^{Z}$ is empty unless $\J=0$, in which case it equals ~$X$.}
so that 
$$X_0=X_0^Z+  X^{X\setminus Z}_0.$$
By the Vogel condition \eqref{vogelcondition},   $H_1$ intersects $X_0^{X\setminus Z}$ properly. Set  
$$
X_{1}=H_1\cdot X^{X\setminus Z}_0
$$
and decompose analogously $X_1$ into the components $X_1^Z$ contained in $Z$
and the remaining components $X_1^{X\setminus Z}$, so that 
$
X_{1}=X^Z_{1}+X_{1}^{X\setminus Z}. 
$
Define inductively $X_{k+1}=H_{k+1}\cdot X_k^{X\setminus Z}$, $X_{k+1}^Z$, and $X_{k+1}^{X\setminus Z}$. 
Then 
$$
V^h:=X^Z_{0}+ X^Z_{1}+\cdots + X^Z_{n}
$$
is the {\it Vogel cycle}\footnote{If $\J$ is the 
pullback to $X$ of the radical sheaf of an analytic set $A$,  
this is precisely Tworzewski's algorithm, \cite{T}. 
The notion Vogel 
cycle was introduced by Massey \cite{Mass1, Mass}. For a generic choice of Vogel sequence 
the associated Vogel cycle 
coincides with the {\it Segre cycle} introduced by Gaffney-Gassler, \cite{GG}, see Lemma 2.2 in \cite{GG}.} associated with the Vogel sequence $h$. 
Let $V^h_k$ denote the components of $V^h$ of codimension $k$, i.e.,  $V^h_k=X_k^Z$.   
The  irreducible components of
$V^h$  that appear in any Vogel cycle,   associated with a generic Vogel sequence    at $x$,  are called 
{\it fixed} components in \cite{GG}. 
The remaining  ones  are called {\it moving}. 
It turns out that the fixed Vogel components of $\J$ coincide with the distinguished 
varieties of $\J$, see, e.g.,  see \cite{GG} or \cite{artikelaswy}. 

It is proved  in \cite{GG} and in \cite{T} 
that the multiplicities  $e_k(\J,X,x):=\mult_x V_k^h$ and 
$m_k(\J,X,x):=\mult_x X_k^{X\setminus Z}$ are
independent of $h$ for a generic $h$, where however ``generic'' depends on $x$,
cf., Remark~\ref{alvar}; 
these numbers are called the 
{\it Segre numbers} and \emph{polar multiplicities}, respectively. 



\begin{remark}\label{alvar}
Recall that if $W$ 
is an analytic cycle in $X$, then the 
Lelong number at $x\in X$ of the current of integration $[W]$ along $W$ is
precisely the multiplicity $\mult_x W$ of $W$ at $x$. 

Assume that   $x$ is  a point for which $n_k(\J,X,x)\ge 1$  for some $k$,
where we use the notation from  Theorem \ref{gking}. Moreover, let
$V^h$ be a generic Vogel cycle such that $\mult_x V^h_k=e_k(x)$. Then 
$V^h_k=S^\J_k+W$, where we have identified $S^\J_k$ in Theorem
\ref{gking} with the corresponding cycle and $W$ is a positive cycle of codimension $k$, such that
$\mult_x W=n_k(\J,X,x)$. Since $n_k(\J,X,y)\ge 1$ only on a set
of codimension $\ge k+1$, at most points $y$ on $V^h_k$ we have that
$e_k(\J,X,y)=\mult_y(S^\J_k)$ and hence $\mult_y V^h_k > e_k(\J,X,y)$.
As soon as there is a  moving component at $x$ it is thus impossible to find  a Vogel cycle 
that realizes the Segre numbers in a whole \nbh of $x$.
\end{remark}


In \cite{artikelaswy} Theorem \ref{gking} with $G$ replaced by $\log
|f|$ was proved by showing that $M_k^f:=\1_Z(dd^c \log |f|)^k$ can be
seen as a certain average (of currents of integration) of Vogel
cycles. The fixed Vogel components then appear as the leading part
$S_k^\J$ in the Siu decomposition of $M_k^f$, 
whereas the remainder term $N_k^f$ is a mean value of the moving parts.

\section{Preliminaries}\label{prelim}

Let $\mu$ be a positive closed current on $X$. Recall that if $W$ is any subvariety,
then $\1_W \mu$ and  $\1_{X\setminus W}\mu$ are 
positive closed currents as well; this is the Skoda-El~Mir theorem,
see, e.g., \cite[Chapter~III.2.A]{Dem}.

\begin{lma}\label{lma2} 
Let  $\mu$ be  a positive closed current of bidegree $(p,p)$ that has
support on a subvariety of codimension $k$. If $k>p$ then $\mu=0$. If
$k=p$, then
$\mu=\alpha_1 [W_1]+\cdots+\alpha_\nu [W_\nu]$ where $W_j$ are the irreducible components
of $W$ and $\alpha_j\ge 0$.
\end{lma}

We refer to the first part of Lemma \ref{lma2} as the  \emph{dimension
  principle}. A proof can be found
in \cite[Chapter~III.2.C]{Dem}.

\smallskip 

If $b$ is psh and locally bounded and $T$ is any positive closed current, then
$T\w(dd^c b)^k$ is a well-defined positive current for any $k$, and
if $b_j$ is a decreasing 
sequence of bounded psh functions converging pointwise to $b$, then
\begin{equation}\label{apa0}
T\w (dd^c b)^k=\lim_j T\w (dd^c b_j)^k, \quad T\w b(dd^c b)^{k}=\lim_j T\w b_j (dd^c b_j)^{k},
\  k\le n.
\end{equation}
See, e.g., \cite[Theorem III.3.7]{Dem}. The  case $T\equiv 1$ was first proved by 
Bedford and Taylor, \cite{BT}.

\begin{prop}\label{apa1} 

Assume that $v,b$ are psh and that
$b$ is (locally) bounded. 

\noindent (i)
For $k\le n-1$, 
$$
v(dd^c b)^{k}
$$
has locally finite mass; more precisely, for any compact sets $L,K$, such that
$L\subset int(K)$, we have
\begin{equation}\label{apa2}
\|v(dd^c b)^{k}\|_L\le C_{K,L}\|v\|_K(\sup_K |b|)^k.
\end{equation}

\noindent (ii)
Moreover, if the unbounded locus of $v$ has Hausdorff dimension
$<2n-1$, then 
\begin{equation}\label{koko}
dd^c(v(dd^c b)^k)=dd^c v\w (dd^c b)^k.
\end{equation} 
If $v_j$ is a decreasing sequence of psh functions converging pointwise to
$v$, then 
\begin{equation}\label{nono}
v_j(dd^c  b)^k \to v(dd^c b)^k,
\end{equation}
and 
\begin{equation}\label{jojo}
dd^cv_j\w (dd^c b)^k\to dd^c v\w (dd^c b)^k
\end{equation}
in the current sense.
\end{prop}


The first part of Proposition \ref{apa1} follows immediately from
Proposition~3.11 in \cite[Chapter~III]{Dem}. 
Moreover, Proposition~4.9 in
loc.\ cit.\ applied to $u_1=v$ and $u_j=b$ implies \eqref{nono} and
\eqref{jojo}. 
If we choose $v_j$ smooth, then  
\[
dd^c(v_j(dd^c b)^k)=dd^c v_j \w (dd^c b)^k. 
\]
Thus \eqref{koko} follows from \eqref{nono} and \eqref{jojo}. 
In fact, the assumption about the Hausdorff dimension is not
necessary; 
an elegant and quite direct  argument 
has been communicated to us by Z.\ B\l{}ocki, \cite{Blocki}.

\begin{cor}\label{pluto}
If $b$ is psh and (locally) bounded on $X$ and $W$ is an analytic variety
of positive codimension, then for each $k\geq 0$, 
\begin{equation}\label{apa3}
\1_W(dd^c b)^k=0.
\end{equation}
\end{cor}

\begin{proof} It is enough to consider the case when $W$ is a smooth hypersurface.
The general case follows by stratification. Since it is a local statement,
we may choose coordinates $z=(z',w)$ so that $W=\{w=0\}$. Notice that in a set
$|w|\le r, |z'|\le r'$, we have that
$\1_W(dd^c b)^k$ is the value at $\lambda=0$ of
$$
-(|w|^{2\lambda}-1)(dd^c b)^k.
$$
Since $|w|^{2\lambda}-1$ is psh, \eqref{apa3} follows from  \eqref{apa2}
since the total mass of 
$|w|^{2\lambda}-1$ tends to $0$ when $\lambda\to 0$.
\end{proof}

\begin{lma}\label{apelsin}
If $b$ is psh and (locally) bounded on $X$ and $i\colon Y\to X$ is a
smooth submanifold, then for $k\le n$, 
\begin{equation}\label{apa4}
[Y]\w(dd^c b)^k=i_*(dd^c i^* b)^k,\quad [Y]\w b (dd^c b)^{k}=i_*\big
(i^*b(dd^c i^* b)^{k}\big).
\end{equation}
\end{lma}

\begin{proof}
First assume that $b$ is smooth. Then
$$
\int_X [Y]\w (dd^c b)^k\w\xi=\int_Y (dd^ci^* b)^k\w i^*\xi=
\int_X i_*\big((dd^ci^* b)^k\big)\w\xi
$$
and similarly 
$$
\int_X [Y]\w b(dd^c b)^k\w\xi=
\int_X i_*\big(i^*b(dd^ci^* b)^k\big)\w\xi,
$$
so that \eqref{apa4} holds in this case. Now let $b$ be
bounded and psh and let $b_j$ be a decreasing sequence
of smooth psh functions converging pointwise to $b$. Now
\eqref{apa4} follows from the smooth case and \eqref{apa0}.
\end{proof}

\section{Higher Monge-Amp\`{e}re products}\label{avfall}

Let $G$ be a psh function of the form \eqref{formen}. 
We will give meaning to 
\begin{equation}\label{tejon}
(dd^c G)^k
\end{equation} by inductively defining it as 
 $(dd^c G)^0=1$ and 
\begin{equation}\label{offer}
(dd^c G)^k:=dd^c\big( G \1_{X\setminus Z}(dd^cG)^{k-1}\big), \quad k\geq 1.
\end{equation}
Proposition \ref{asgam} below asserts that this definition makes sense
and that $(dd^c G)^k$ are positive and closed. 
As pointed out in the introduction this definition coincides with the iterative definition \eqref{forvirrad} for $k\leq p$. 


\begin{prop}\label{asgam}
Let $X$ be a complex manifold of dimension $n$, let $f$ be a tuple
of global functions of $X$, let $G$ be a psh function
of the form \eqref{formen}, and 
 let $G_j$ be a decreasing sequence of 
smooth psh functions in $X$  converging pointwise to $G$.
Assume that \eqref{tejon} is inductively defined via \eqref{offer} for
a fixed $k$. Then
$$
G\1_{X\setminus Z}(dd^c G)^{k}:=\lim_j G_j\1_{X\setminus Z}(dd^c G)^{k}
$$
has locally finite mass and does not depend on the choice of sequence
$G_j$. 
Moreover $(dd^c G)^{k+1}=dd^c(G\1_{X\setminus Z}(dd^c G)^{k})$ is
positive and closed. 
\end{prop}

The proof below relies heavily on the fact that
$G$ is of the form \eqref{formen}. It could be interesting to
investigate whether Proposition \ref{asgam} holds for a wider class of
psh functions $G$ with unbounded locus $Z$.


\begin{proof}
Let $\pi\colon \widetilde X\to X$ be a smooth modification such that
$\pi^*\J$ is principal and its divisor
is of the form
\begin{equation}\label{deco}
D=\sum \alpha_j D_j,
\end{equation}
where $D_j$ are smooth hypersurfaces with normal crossings.
In  particular, then $\pi^* f=f^0 f'$, where
$f^0$ is a section of the line bundle $L_D$ that defines $D$ and
$f'$ is a non-vanishing tuple of sections of $L_D^{-1}$.

Locally on $\widetilde X$ we can choose a frame for $L_D$ and
in this frame we have, cf. \eqref{formen}, 
\begin{equation}\label{julbestyr}
\pi^* G=\log|f^0|+\log|f'|+\pi^* h=:\log|f^0|+b.
\end{equation} 
Since $\log|f^0|$ is pluriharmonic outside $$|D|:=\cup_j D_j$$ it follows that 
$$
b=\log|f'|+\pi^* h
$$
is psh there; 
furthermore it is locally bounded at $|D|$. By a standard argument
$b$ has a unique (bounded) psh extension $B$ across $|D|$. 
Notice that $dd^c B$ is a global positive closed
current on $\widetilde X$ and
\begin{equation*}
dd^c\pi^* G=[D]+dd^c B.
\end{equation*}


\smallskip
Let $G_j$ be a decreasing sequence of smooth psh functions converging pointwise
to  $G$.  Since 
$$
dd^c G_j=\pi_*(dd^c\pi^* G_j)\to \pi_*\big(dd^c \pi^* G\big)
=\pi_*\big([D]+dd^c B\big)
$$
it follows that
\begin{equation*}
dd^c G=\pi_*\big([D]+dd^c B\big).
\end{equation*}

Let us now assume that we have proved Proposition~\ref{asgam} as well as
the equality 
\begin{equation}\label{orm1}
\big(dd^c G\big)^{\ell}=\pi_*\big([D]\w(dd^c B)^{\ell-1}+(dd^c B)^\ell\big)
\end{equation}
for $\ell\leq k$. We are to see that then: 
\smallskip

\noindent (i) \emph{$G\1_{X\setminus Z}(dd^cG)^{k}:=\lim_j G_j\1_{X\setminus Z}(dd^cG)^{k}$ has locally finite
mass.}

\smallskip

\noindent (ii) \emph{If 
\begin{equation*}
(dd^cG)^{k+1}:=dd^c\big(G\1_{X\setminus Z}(dd^c G)^{k}\big),
\end{equation*}
then \eqref{orm1} holds for $\ell=k+1$.}

\smallskip
\noindent As soon as (i) and (ii) are verified, Proposition~\ref{asgam} follows.
\smallskip

Notice that if $\mu$ is a closed positive current, then
\begin{equation}\label{orm2}
\1_Z\pi_*\mu=\pi_*(\1_{|D|}\mu).
\end{equation}
In view of Corollary~\ref{pluto} we have that
\begin{equation}\label{orm3}
\1_{|D|}(dd^c B)^{k}=0.
\end{equation}
From the induction hypothesis \eqref{orm1}, \eqref{orm2} and \eqref{orm3} 
we  get
\begin{equation}\label{orm4}
\1_{X\setminus Z}\big(dd^c G\big)^{k}=
\pi_* (dd^c B)^{k}.
\end{equation}
By Proposition~\ref{apa1}, $(\pi^*G)(dd^c B)^{k}$ has locally finite mass, and
$$
(\pi^* G_j)(dd^c B)^{k}\to (\pi^*G)(dd^c B)^{k}
$$
if $G_j$ is any decreasing sequence of psh functions that tends to $G$.
If $G_j$ are smooth we have by \eqref{orm4} that
$$
G_j \1_{X\setminus  Z}(dd^c G)^{k}=\pi_*\big((\pi^* G_j)(dd^c B)^{k}\big),
$$
which tends to 
\begin{equation}\label{mackmyra}
G\1_{X\setminus Z}(dd^c G)^{k}=\pi_*\big( (\pi^* G)(dd^c B)^{k}\big),
\end{equation}
which has locally finite mass. 
Thus (i) is verified.

We now consider (ii).  We claim that
\begin{equation}\label{gras}
dd^c \big(\pi^*G \w (dd^c B)^{k}\big)=[D]\w (dd^c B)^{k}+(dd^c B)^{k+1}.
\end{equation}
Recall that locally $\pi^*G= v +B$, where $v=\log|f^0|$ and $B$ is psh and bounded.
Take smooth psh $v_j$ that decrease to $v$.
Then $v_j+B$ are psh and decrease to $v+B$ and thus, by Proposition~\ref{apa1},
$$
v_j(dd^c B)^{k}+ B(dd^c B)^{k}=
(v_j+B) (dd^c B)^{k}\to (v+B) (dd^c B)^{k}.
$$
It follows that 
\begin{equation*}
(v+B)(dd^c B)^{k}=v(dd^c B)^{k}+B(dd^c B)^{k}.
\end{equation*}
From  Proposition~\ref{apa1} we get that
$$
dd^c\big(v(dd^cB)^{k}\big)=[D]\w(dd^c B)^{k},
$$
which proves the claim. 
In view of \eqref{mackmyra} and \eqref{gras}  the statement (ii) now follows.
\end{proof}

For future reference we notice  that
\begin{equation}\label{strut}
M^\J_k=\pi_*\big([D]\w(dd^c B)^{k-1}\big), \quad \1_{X\setminus
  Z}(dd^c G)^k=\pi_* (dd^c B)^k.
\end{equation}
In fact $\1_{X\setminus Z}(dd^c G)^k$ equals the \emph{non-pluripolar
  product} $\langle dd^c G\rangle^k$ as defined in \cite{BT2, BEGZ}.

\smallskip 

It follows from the proof above and Proposition~\ref{apa1} that if
 $G_j$ is any decreasing sequence of psh functions converging pointwise to $G$, 
then $G_j\1_{X\setminus Z}(dd^c G)^{k-1}\to G\1_{X\setminus Z}(dd^c G)^{k-1}$
and 
$$
dd^c (G_j\w \1_{X\setminus Z}(dd^cG)^{k-1})=
dd^c G_j\w \1_{X\setminus Z}(dd^cG)^{k-1}\to (dd^c G)^k.
$$

Recall that if $G_j$ are psh functions that decrease to
$G$, then 
\begin{equation*}
\lim_j (dd^c G_j)^k = (dd^c G)^k, \quad k\leq p,
\end{equation*}
see, e.g., \cite[Proposition ~III.4.9]{Dem}. 
However, for $k>p$ one cannot hope for a definition of $(dd^c G)^k$
that is robust in this sense. 
%
In fact, even if $G_j$ and $\widetilde G_j$ are
sequences of smooth psh functions decreasing to $G$ and $(dd^c G_j)^k$
and $(dd^c \widetilde G_j)^k$ converge to
positive closed currents $T$ and $\widetilde T$, respectively, $T$
might be different from $\widetilde T$, as is illustrated by the
following example. 

\begin{ex}\label{exodus} 
Let $\varphi=(w,zw)$. 
Then 
\[
dd^c\log |\varphi|=dd^c\log |w|+dd^c\log (1+|z|^2)^{1/2}=
[w=0]+dd^c \alpha, 
\]
where $[w=0]$ denotes the current of integration along $\{w=0\}$ and  $\alpha=\log (1+|z|^2)^{1/2}$.  
Thus by \eqref{offer}, 
\[
(dd^c \log |\varphi|)^2=[w=0]\wedge dd^c \alpha. 
\]
Let 
$G_\epsilon=\log (|\varphi|^2 +\epsilon)^{1/2}$ and 
$\widetilde G_\epsilon=\log (|w|^2 +\epsilon)^{1/2} +\alpha$. 
Then $G_\epsilon$ and $\widetilde G_\epsilon$ are smooth psh functions
that decrease
towards $\log |\varphi|$ as $\epsilon$ tends to $0$. 
On the one hand, by \eqref{somna},  
\[
\lim_{\epsilon\to 0} (dd^c G_\epsilon)^2=(dd^c\log |\varphi|)^2.
\]
On the other hand, again using \eqref{somna}, but now for 
$(dd^c \log |w|)^2$, 
\begin{equation*}
(dd^c \widetilde G_\epsilon)^2=
(dd^c \log (|w|^2 +\epsilon)^{1/2})^2 +2 dd^c \log(|w|^2 +\epsilon)^{1/2}\w dd^c
  \alpha  
\longrightarrow 
2[w=0]\wedge dd^c \alpha. 
\end{equation*}
\end{ex}

\begin{remark}\label{forsvinn} 
Assume that $X_\ell$ is an exhaustion of $X$ by relatively compact
subsets such that the restriction $\J_\ell$ of $\J$ to $X_\ell$ is
generated by global bounded functions. It would be interesting
to know whether, or under what assumptions, the currents $M_k^{\J_\ell}$
then converge. 
Convergence would give us a global canonical representation of the Segre numbers
of $\J$. 


Assume that $\J$ is indeed generated by global bounded functions  
and let $G_\ell$ denote the Green function with poles along $\J_\ell$. 
 Then, arguing as in the proof of Proposition \ref{asgam} and using
the notation from that proof, 
\[
\pi^*G_\ell=\log |f^0|+ B_\ell,
\]
where $B_\ell$ is psh and bounded, and moreover
\[
(dd^c G_\ell)^k=\pi_*([D]\wedge (dd^c B_\ell)^{k-1} + (dd^c B_\ell)^k). 
\]
Assume that $G_\ell$
decrease towards $G$. Then $B_\ell$ decrease towards $B$, as defined in
\eqref{julbestyr}, and thus $\lim_\ell(dd^c
G_\ell)^k=(dd^c G)^k$ in light of \eqref{apa0} and \eqref{orm1}. 
\end{remark}


\section{Lelong numbers}\label{kjol}

Let $T$ be  a positive closed $(k,k)$-current. 
If $k=n$, following
\cite[Section~5]{artikelaswy}, we let 
$$
M^\xi_0\w T:=\1_{\{x\}} T.
$$
Otherwise  
\begin{equation*}
M^\xi_{n-k} \w T:=\1_{\{x\}}\big((dd^c\log|\xi|)^{n-k}\w T\big);
\end{equation*}
here we inductively define
\begin{multline*}
(dd^c\log|\xi|)^{\ell}\w T:=\\
dd^c\big(\log|\xi|\w(dd^c\log|\xi|)^{\ell-1}\w T\big)=
\lim_jdd^c\big(v_j\w(dd^c\log|\xi|)^{\ell-1}\w T\big),
\end{multline*}
where $v_j$ is a decreasing sequence of smooth psh functions converging
pointwise to $\log|\xi|$. Because of the dimension principle it is not
necessary to insert $\1_{X\setminus\{x\}}$ in this definition, cf., 
Section~\ref{avfall}. See Remark~\ref{brud} below for another possible
definition of $M^\xi_{n-k} \w T$. 
Clearly $M^\xi_{n-k}\w T$  is an  $(n,n)$-current with 
support at $x$, and it is in fact equal to
$\alpha [x]$, where $\alpha$ is the Lelong number of $T$ at $x$, see, e.g,
\cite[Lemma~2.1]{artikelaswy}.

\begin{remark}\label{brud}
As is pointed out in \cite[Section~5]{artikelaswy} one can
define 
$
M^{\xi}\w T
$
as the value at $\lambda=0$ of the current-valued analytic function
$$
\lambda\mapsto
\frac{\dbar|\xi|^{2\lambda}\w\partial|\xi|^2}{2\pi i|\xi|^2}\w(dd^c\log|\xi|)^{n-k-1}\w T.
$$
\end{remark}

\section{Proof of Theorem~\ref{gking}}\label{bevisas} 
We will prove the slightly more general formulation of Theorem
\ref{gking} stated in Remark \ref{gulpenna}, i.e., we let $G$ be
any psh function of the form \eqref{formen}.

We still assume that $\pi\colon  \widetilde X\to X$ is a smooth modification
and use  the notation from the proof of 
Proposition~\ref{asgam}.
Notice that $L_D$ has a Hermitian  metric such that 
$|f^0|_{L_D}=|\pi^*f|$. 
By the Poincar\'e-Lelong formula, 
\begin{equation}\label{torsk1}
dd^c\log|\pi^* f|=[D]+\omega_f,
\end{equation}
where $\omega_f$ is the first Chern form for $L^{-1}_D$. 

Let us  fix a local holomorphic frame so that $\log|f'|$ is a 
well-defined function as above.
Since 
\[
\log |\pi^* f|=\log |f^0|+\log |f'|,
\]
from \eqref{torsk1} we have that 
\begin{equation}\label{vinterlov}
\omega_f=dd^c\log|f'|.
\end{equation}
Let $b$ be the psh bounded function outside $|D|$ defined in \eqref{julbestyr}. 
If we choose another local frame for $L_D$, then $\log|f'|$ is changed
to $\log |f'| + \alpha$ where $\alpha$ is pluriharmonic, and $b$ is
thus 
changed to $\tilde
b:= b+\alpha$. Moreover $\widetilde B:=B+\alpha$ is the unique psh
extension of $\tilde b$ across $|D|$, cf.\ the proof of Proposition \ref{asgam}. It follows that $A$, locally
defined as
\begin{equation}\label{elefant}
A:=B-\log |f'|,
\end{equation} 
is a global upper semicontinuous extension of $\pi^*h$ across $|D|$. 
Notice also that $A(dd^c B)^\ell$ 
is well-defined on $\widetilde X$
and, in light of \eqref{vinterlov} and \eqref{elefant}, that
\begin{equation*}
(dd^c B)^{k-1}-\omega_f^{k-1}=dd^c\Big( A\sum_{\ell=0}^{k-2} (dd^c B)^\ell\w \omega_f^{k-2-\ell}\Big).
\end{equation*}
Assume now that $Y\subset \widetilde X$ is a smooth submanifold and
that $i\colon Y\to \widetilde X$ is the natural inclusion. Then $i^*B$ is psh and 
bounded, $i^*\log|f'|$ is smooth, and,  in the same way as above,
 $i^*A$ is a global upper semi-continuous function on $Y$ and
 \begin{equation}\label{lussekatt2}
(dd^c i^*B)^{k-1}-i^*\omega_f^{k-1}=dd^c\Big( i^*A\sum_{\ell=0}^{k-2} (dd^c i^*B)^\ell\w i^*\omega_f^{k-2-\ell}\Big).
\end{equation}
In view of Lemma~\ref{apelsin}, \eqref{lussekatt2} implies that 
\begin{equation*}
[Y]\w\Big ((dd^c B)^{k-1}-\omega_f^{k-1}\Big )=dd^ci_*\Big( i^*A\sum_{\ell=0}^{k-2} (dd^c i^*B)^\ell\w i^*\omega_f^{k-2-\ell}\Big).
\end{equation*}


\smallskip

The currents $(dd^c \log |f|)^k$ and $M^f_k$ are defined in a completely
analogous way as $(dd^c G)^k$ and 
$M^\J_k$,  just replacing $G$ by $\log|f|$, cf., the introduction and
the end of Section \ref{segretal} and also \cite{artikelaswy}.
Arguing as in the proof of Proposition~\ref{asgam}, we get, cf., \eqref{strut},   that
\begin{equation*}
M^f_k=\pi_*([D]\w\omega_f^{k-1}), \quad 
\1_{X\setminus Z}(dd^c\log |f|)^k=\pi_*\omega_f^{k}
\end{equation*}

\begin{lma}\label{torsdag}
The currents $M^\J_k$ and $M^f_k$ have the same
Lelong number at each point $x\in X$. 
Moreover, the currents $\1_{X\setminus Z} (dd^c G)^k$ and
$\1_{X\setminus Z} (dd^c \log |f|)^k$ have the same Lelong number at
each point $x\in X$. 
\end{lma}

\begin{proof}
Let us fix a point $x\in X$ and let $\xi$ be a tuple of functions that defines
the maximal ideal $\m_x$ at $x$.
We can choose the modification $\pi\colon\widetilde X\to X$
so that also $\pi^* \m_x$ is principal, i.e.,
$\pi^*\xi=\xi^0\xi'$, where $\xi^0$ is a section of a line bundle $L_E$ that
defines the exceptional divisor $E$,
and $\xi'$ is a non-vanishing tuple of sections of $L_E^{-1}$. Let us assume that 
\begin{equation}\label{rosa} 
E=\sum_\kappa \beta_\kappa E_\kappa,
\end{equation}
where $E_\kappa$ are  irreducible with simple normal crossings and $\beta_\kappa$
are integers. 
We may also assume that, for each $j$, cf., \eqref{deco},
either $D_j\subset |E|$ or all
$E_\kappa$ intersect $D_j$  properly and that  
\[E^{D_j}_\kappa:=E_\kappa\cap D_j\] are  smooth.
Let $\omega_\xi$ be the first Chern form
of $L_E^{-1}$ with respect to the metric induced by $\xi$,
so that
$$
\omega_\xi=dd^c\log|\xi'|,
$$
cf., \eqref{vinterlov},  
and
$$
dd^c\log|\pi^*\xi|=[E]+\omega_\xi.
$$

\smallskip

Let $i_j\colon D_j\to \widetilde X$ be the injection of  $D_j$ as a submanifold of
$\widetilde X$. It follows from \eqref{deco}, 
\eqref{strut} and Lemma~\ref{apelsin}  that
\begin{equation}\label{robot}
M^\J_k=\sum_j\alpha_j \pi_*(i_j)_*\big((dd^c (i_j)^*B)^{k-1}\big).
\end{equation} 
In order to prove the first part of the lemma, it is enough to consider one single term
in \eqref{robot} and verify that
$$
T^\J_k:=\pi_*i_*\big((dd^c i^*B)^{k-1}\big)
$$
and
$$
T^f_k:=\pi_*i_*\big(i^*\omega_f^{k-1}\big)
$$
have the same Lelong numbers, where we write $D=D_j$ and $i=i_j$ for simplicity.

\smallskip
Let us first assume that $k=n$. If $D\subset|E|$, then $T_n^\J$ and $T_n^f$ both have
support at $x$. 
In view of \eqref{lussekatt2},  with $Y=D$,  we  have that 
$$
T^\J_k-T^f_k=
dd^c \pi_*i_* \Big( i^*A\sum_{\ell=1}^{k-2} (dd^c i^*B)^\ell\w i^*\omega_f^{k-2-\ell}\Big)=:dW,
$$
where $W$ has support at $x$.
By Stokes' theorem thus 
$$
\int (T_n^\J-T_n^f)=\int dW=0,
$$ 
which means that $T_n^\J$ and $T_n^f$ have the same Lelong number at $x$. If
$D$ is not contained in $E$, then 
$i^{-1}E$ has positive codimension in $D$ and therefore,
$$
\1_{\{x\}}T_n^\J=\pi_*i_*(\1_{|i^{-1}E|}(dd^c i^* B)^{n-1})=0
$$
by Corollary~\ref{pluto}. In the same way we see that $\1_{\{x\}}T_n^f=0$.

\smallskip
Let us now assume that $k<n$. If $D\subset|E|$, then $T^\J_k$ and
$T^f_k$ are positive closed $(k,k)$-currents with support at $x$, so by the dimension principle they both vanish.
We can therefore assume that 
$i^*\pi^*\xi$ does not vanish identically on $D$; by assumption 
it then defines a smooth divisor $E^D$ on $D$.
Locally  on $D$,
$$
\log|i^*\pi^*\xi|=\log|i^*\xi^0|+\log|i^*\xi'|,
$$
and thus 
\begin{equation}\label{utrymme}
dd^c\log|i^*\pi^*\xi|= [E^D]+i^*\omega_\xi,
\end{equation}
where $[E^D]$  is the  Lelong current on $D$ associated to $E^D$.
If $v_j$ are as in Section~\ref{kjol}, then 
\begin{equation*}
dd^ci^*\pi^* v_j\to [E^D]+i^*\omega_\xi.
\end{equation*}
Now
$$
dd^c(v_j T^\J_k)=\pi_*i_*\big(dd^c i^*\pi^* v_j\w (dd^c i^* B)^{k-1}\big)
$$
so that
$$
dd^c\log|\xi|\w T^\J_k=\pi_*i_*\big(([E^D]+i^*\omega_\xi)\w(dd^c i^* B)^{k-1}\big)
$$
by Proposition ~\ref{apa1} and \eqref{utrymme}. 
Moreover, since $\pi_*i_* \big([E^D]\w(dd^c i^*
B)^{k-1}\big )$ has support at $x$, by the dimension principle,  
$$
dd^c\log|\xi|\w T^\J_k=\pi_*i_*\big(i^*\omega_\xi\w(dd^c i^* B)^{k-1}\big).
$$
By induction we get
$$
(dd^c\log|\xi|)^{n-k}\w T^\J_k=
\pi_*i_*\big(([E^D]+i^*\omega_\xi)\w i^*\omega_\xi^{n-k-1}\w(dd^c i^* B)^{k-1}\big).
$$
Therefore, by Corollary ~\ref{pluto}, 
$$
M^\xi_{n-k}\w T^\J_k=
\1_{\{x\}}(dd^c\log|\xi|)^{n-k}\w T^\J_k=
\pi_*i_*\big([E^D]\w i^*\omega_\xi^{n-k-1}\w(dd^c i^* B)^{k-1}\big).
$$
Let $\iota_\kappa\colon E_\kappa^D\to  D$ be the natural injection.
By \eqref{rosa} and Lemma~\ref{apelsin} we have that 
$$
M^{\xi}_{n-k}\w T^\J_k=\sum_\kappa \beta_\kappa
\pi_*i_*(\iota_\kappa)_*\big(
(\iota_\kappa)^*i^*\omega_\xi^{n-k-1}\w (dd^c(\iota_\kappa)^*i^* B)^{k-1}\big).
$$
By analogous arguments, 
$$
M^{\xi}_{n-k}\w T^f_k=\sum_\kappa \beta_\kappa
\pi_*i_*(\iota_\kappa)_*\big(
(\iota_\kappa)^*i^*\omega_\xi^{n-k-1}\w (\iota_\kappa)^*i^* \omega_f^{k-1}\big).
$$
For simplicity in notation let us assume that $E^D$ has just one
irreducible component and let $\iota:E^D\to D$ be the natural
injection. 
By \eqref{lussekatt2} applied to $E^D$ we have that 
\begin{multline*}
M^{\xi}_{n-k} \w T^\J_k-M^{\xi}_{n-k} \w T^f_k=\\
dd^c\pi_*i_*\iota_* \Big(\iota^* i^*A ~\iota^*i^*\omega_\xi^{n-k-1}\w
 \sum_{\ell =0}^{k-2}(dd^c\iota^*i^* B)^{\ell}
\iota^*i^* \omega_f^{k-1-\ell}\Big)= :dW,
\end{multline*}
where $W$ has support at $x$.
It follows by Stokes' theorem that the integral of this current is zero,
and thus the Lelong numbers at $x$ of $T_k^\J$ and $T_k^f$ coincide.
Thus the first part of the lemma is proved.

By analogous arguments we get that 
$\pi_*(dd^cB)^k$ and $\pi_*(\omega_f)^k$ have the same Lelong
number at $x$, which
proves the second part of the lemma, cf.\ \eqref{strut} and \eqref{rosa}. 



\end{proof}

We can now conclude  the proof of Theorem~\ref{gking}.

\begin{proof}[Proof of Theorem~\ref{gking}]

Let $D^\ell_j$ be the irreducible components of $D$ such that
$\pi(D^\ell_j)$ have codimension $\ell$.  Then 
$$
M^\J_k=\pi_*\big([D]\w(dd^c B)^{k-1}\big)=
\pi_*\big(\sum_{\ell\le k}\sum_j ([D^\ell_j]\w
(dd^c B)^{k-1}\big)
$$
since terms with $\ell>k$ vanish because of the dimension principle.
We claim that
\begin{equation}\label{skolgatan}
M^\J_k=
\pi_*\big(\sum_j ([D^k_j]\w
(dd^c B)^{k-1}\big)+
\pi_*\big(\sum_{\ell<k}\sum_j ([D^\ell_j]\w
(dd^c B)^{k-1}\big)=: S_k^\J+N^\J_k
\end{equation}
is the Siu decomposition of $M^\J_k$. 
First notice that since
$$
\pi_*\big([D^k_j]\w(dd^c B)^{k-1}\big)
$$
is a $(k,k)$-current with support on the set $Z:=\pi(D^k_j)$ of codimension
$k$ it must be of the form $\alpha [Z]$ where $\alpha$ is a constant,
see Lemma~\ref{lma2}.

It is now enough to see that 
if $W$ is a subvariety of codimension $k$, then
$\1_WN^\J_k=0$, i.e., 
\begin{equation*}
\1_W \pi_*\big([D^\ell_j]\w(dd^c B)^{k-1}\big)=0
\end{equation*}
if $\ell<k$. Let $i\colon D^\ell_j\to\widetilde X$ be the natural injection.
By Lemma ~\ref{apelsin} we have
$$
\1_W \pi_*\big([D^\ell_j]\w(dd^c B)^{k-1}\big)=
\1_W(\pi_*i_*\big(dd^c i^*B)^{k-1}\big)=
\pi_*i_*\big(\1_{(\pi\circ i)^{-1}(W)}(dd^c i^* B)^{k-1}\big).
$$
Notice  that since $\pi (D^\ell_j)$ is irreducible and
not contained in $W$ it follows that
$\pi^{-1}(W)\cap D^\ell_j$ has  positive
codimension in $D^\ell_j$, and hence
$\1_{(\pi\circ i)^{-1}(W)}(dd^c i^* B)^{k-1}=0$
in view of Corollary~\ref{pluto}.

Thus \eqref{skolgatan}  is the Siu decomposition. 
Since $M^\J_k$ and $M^f_k$ have the same
Lelong number at each point by Lemma ~\ref{torsdag} and the set where
$N_k^\J$ and $N_k^f$ have positive Lelong number have codimension $>k$
we conclude that $S^\J_k=S^f_k$, see Remark ~\ref{alvar}. 
Since also $\1_{X\setminus Z} (dd^cG)^k$ and 
$\1_{X\setminus Z} (dd^c\log |f|)^k$ have the same Lelong numbers at
$x$ by Lemma \ref{torsdag}, Theorem~\ref{gking} follows from the 
analogous result, Theorem~1.1, for $M^f$ in \cite{artikelaswy}.

\end{proof}

\def\listing#1#2#3#4#5#6{ {\sc #1}  {\it #2} #3
{\bf#4} {#5} #6}

\end{document}